\pdfoutput=1
\RequirePackage{ifpdf}
\ifpdf 
\documentclass[pdftex]{sigma}
\else
\documentclass{sigma}
\fi

\numberwithin{equation}{section}

\newtheorem{Theorem}{Theorem}[section]
\newtheorem{cor}[Theorem]{Corollary}
\newtheorem{lem}[Theorem]{Lemma}

\newtheorem{thm}[Theorem]{Theorem}

 { \theoremstyle{definition}
\newtheorem{defi}[Theorem]{Definition}
\newtheorem{exa}[Theorem]{Example}
\newtheorem{rem}[Theorem]{Remark} }

\newtheorem{thmRoman}{Theorem}

\newcommand{\N}{\mathbf{N}}
\DeclareMathOperator{\Irr}{Irr}
\newcommand{\fS}{{\mathfrak{S}}}

\begin{document}

\allowdisplaybreaks

\newcommand{\arXivNumber}{1705.08655}

\renewcommand{\thefootnote}{}

\renewcommand{\PaperNumber}{070}

\FirstPageHeading

\ShortArticleName{Restriction of Odd Degree Characters of $\mathfrak{S}_n$}

\ArticleName{Restriction of Odd Degree Characters of $\boldsymbol{\mathfrak{S}_n}$\footnote{This paper is a~contribution to the Special Issue on the Representation Theory of the Symmetric Groups and Related Topics. The full collection is available at \href{https://www.emis.de/journals/SIGMA/symmetric-groups-2018.html}{https://www.emis.de/journals/SIGMA/symmetric-groups-2018.html}}}

\Author{Christine BESSENRODT~$^\dag$, Eugenio GIANNELLI~$^\ddag$ and J{\o}rn B.~OLSSON~$^{\S}$}

\AuthorNameForHeading{C.~Bessenrodt, E.~Giannelli and J.B.~Olsson}

\Address{$^\dag$~Institute for Algebra, Number Theory and Discrete Mathematics,\\
 \hphantom{$^\dag$}~Leibniz Universit\"at Hannover, Welfengarten 1, D-30167 Hannover, Germany}
\EmailD{\href{mailto:bessen@math.uni-hannover.de}{bessen@math.uni-hannover.de}}

\Address{$^\ddag$~Department of Pure Mathematics and Mathematical Statistics, University of Cambridge, \\
\hphantom{$^\ddag$}~Cambridge CB3 0WA, United Kingdom}
\EmailD{\href{mailto:eg513@cam.ac.uk}{eg513@cam.ac.uk}}

\Address{$^{\S}$~Department of Mathematical Sciences, University of Copenhagen,\\
\hphantom{$^{\S}$}~DK-2100 Copenhagen \O, Denmark}
\EmailD{\href{mailto:olsson@math.ku.dk}{olsson@math.ku.dk}}

\ArticleDates{Received May 25, 2017, in f\/inal form August 30, 2017; Published online September 05, 2017}

\Abstract{Let $n$ and $k$ be natural numbers such that $2^k < n$. We study the restriction to $\mathfrak{S}_{n-2^k}$ of odd-degree irreducible characters of the symmetric group~$\mathfrak{S}_n$. This analysis completes the study begun in~[Ayyer~A., Prasad~A., Spallone~S., \textit{S\'em. Lothar. Combin.} \textbf{75} (2015), Art.~B75g, 13~pages] and recently developed in~[Isaacs~I.M., Navarro~G., Olsson~J.B., Tiep~P.H., \textit{J.~Algebra} \textbf{478} (2017), 271--282].}


\Classification{20C30; 05A17}

\section{Introduction}\label{sec:intro}
Let $n$ be a natural number, and let $\chi$ be an irreducible character of odd degree of the symmetric group $\fS_n$. Then there exists a unique odd-degree irreducible constituent of the restriction $\chi_{\fS_{n-1}}$. This interesting fact was discovered recently in~\cite{APS}. The result had immediate applications in the study of natural correspondences of characters of f\/inite groups (see for example~\cite{GKNT}). In \cite[Theorem~A]{INOT} the result mentioned above was generalized, by showing that given any $k\in\mathbb{N}$ such that $2^k < n$, there exists a unique odd-degree irreducible constituent $f_k^n(\chi)$ of $\chi_{\fS_{n-2^k}}$ appearing with odd multiplicity. The main goal of this article is to study for all $n,k\in\mathbb{N}$ the map
\begin{gather*}
f_k^n\colon \ \Irr_{2'}(\fS_n)\longrightarrow\Irr_{2'}(\fS_{n-2^k}),
\end{gather*}
naturally def\/ined by Theorem~A of~\cite{INOT}. All our results are proved using a description of $f_k^n$ in terms of the natural partition labels of the involved irreducible characters.

Before describing the main results of this paper, we introduce some vocabulary. If~$2^k$ appears in the binary expansion of $n$ we say that~$2^k$ is a {\it binary digit} of~$n$. Similarly we say that two natural numbers~$m$ and~$n$ are {\it $2$-disjoint} if they do not have any common binary digit. On the other hand, if $m\le n$ and all the binary digits of $m$ appear in the binary expansion of~$n$, then we say that $m$ is a {\it binary subsum} of~$n$. This will be
denoted by $m\subseteq_2 n$. Let $\nu_2(n)$ be the exponent of the highest power of~2 dividing the integer~$n$.

A question raised in \cite{INOT} may be phrased as: \textit{For which $n$ and $k$ is $f^n_k$ surjective?} The authors showed that $f^n_k$ is surjective whenever $2^k$ is a binary digit of~$n$, and they observed that otherwise~$f^n_k$ could be both surjective or not (see \cite[Proposition~4.5 and Remark~4.6]{INOT}). In this paper we answer the question of surjectivity completely with the following result.

\begin{thmRoman}\label{TheoremA}
Let $n\in\mathbb{N}$, $k\in\mathbb{N}_0$ be such that $2^k<n$. Let $ d(n,k)=\nu_2\big(\big\lfloor \frac{n}{2^k}\big\rfloor\big)$.
\begin{itemize}\itemsep=0pt
\item If $k=0$ then $f^n_k$ is surjective if and only $d(n,k) \le 2$.
\item If $k>0$ then $f^n_k$ is surjective if and only $d(n,k) \le 1$.
\end{itemize}
\end{thmRoman}

Theorem~\ref{TheoremA} is a consequence of Theorem~\ref{thm5:image:reform} below, which describes the images of the maps~$f^n_k$.

For all $n\in\mathbb{N}$, $k\in\mathbb{N}_0$ with $2^k<n$ and any $\psi\in\Irr_{2'}(\fS_{n-2^k})$ we def\/ine
the set
\begin{gather*}
\mathcal{E}\big(\psi,2^k\big)=\big\{\chi\in\Irr_{2'}(\fS_n)\, |\, f_k^n(\chi)=\psi\big\},
\end{gather*}
and set $e\big(\psi,2^k\big)=\big|\mathcal{E}\big(\psi,2^k\big)\big|$. We show in Corollary~\ref{cor:reg} that the maps $f^n_k$ are regular on their images. This means that for any~$\psi$ in the image of~$f^n_k$, the number~$e(\psi, 2^k)$ depends only on $n$ and $k$ and not on the specif\/ic~$\psi$. We also give a complete description of those $\psi\in\Irr_{2'}(\fS_{n-2^k})$ such that $e(\psi,2^k)=0$, in Theorem~\ref{thm5:image:reform}.

In the f\/inal part of the paper we study commutativity. For convenience, we sometimes denote $f_k^n$ just by $f_k$, when the natural number $n$ is clear from the context. Then, for $k,{\ell}\in\mathbb{N}_0$, $k<\ell$, such that $2^k+2^{\ell}\leq n$, we may ask: \textit{when is $f_kf_{\ell}=f_{\ell}f_k$?} or more specif\/ically: \textit{when is $f^{n-2^{\ell}}_kf^n_{\ell}=f^{n-2^k}_{\ell}f^n_k$?} In \cite[Proposition~4.3]{INOT} it was proved that $f_kf_{\ell}=f_{\ell}f_k$ whenever $2^{\ell}<n<2^{\ell+1}$. This is the case $\ell=t$ in our second main result, which answers the question completely.

\begin{thmRoman}\label{TheoremB} Let $n=2^t+m$ where $0 \le m < 2^t$. Suppose that $k$, ${\ell}$ satisfy $0 \leq k<{\ell} \leq t$ and $2^k+2^{\ell}\leq n$. Then, {\em with the exception of the case} $n=6$, $k=0$, ${\ell}=1$,
\begin{gather*}
\text{$f_kf_{\ell}= f_{\ell}f_k$ if and only if $2^k > m$ or $\ell =t$}.
\end{gather*}
\end{thmRoman}

\section{Notation and background}\label{background}
Let $n$ be a natural number. We let $\Irr(\fS_n)$ denote the set of irreducible characters of~$\fS_n$ and~$\mathcal{P}(n)$ the set of partitions of $n$. The notation $\lambda\in\mathcal{P}(n)$ is sometimes replaced by $\lambda\vdash n$ and we write~$|\lambda|=n$. There is a natural correspondence $\lambda \leftrightarrow \chi^{\lambda} $ between $\mathcal{P}(n)$ and $\Irr(\fS_n)$. We say then that $\lambda$ labels $\chi^{\lambda}$.
We denote by $\Irr_{2'}(\fS_n)$ the set of irreducible characters of $\fS_n$ of odd degree. If $\chi^{\lambda} \in \Irr_{2'}(\fS_n)$ we say that $\chi^{\lambda}$ is an {\it odd character}, we call $\lambda$ an {\it odd partition} of~$n$ and write $\lambda\vdash_o n$.
Also the empty partition will be considered as an odd partition.

\begin{rem}\label{rem:oddhook} Let $n,k$ be such that $2^k<n$. In \cite[Theorem~A and Proposition~4.2]{INOT} it is shown that the map $f_k^n\colon \Irr_{2'}(\fS_n)\to\Irr_{2'}(\fS_{n-2^k})$ may be described in terms of the odd partitions labelling the odd characters as follows:
\begin{gather*}
f_k^n(\chi^{\lambda})=\chi^{\mu} \Leftrightarrow \mu \vdash_o n-2^k \text{ can be obtained from } \lambda \vdash_o n \text{ by removing a } 2^k \text{-hook}.
\end{gather*}
Correspondingly we write (by abuse of notation) $f_k^n(\lambda)=\mu$. In fact when $\lambda$ is odd, there is only one $2^k$-hook of $\lambda$
whose removal leads again to an odd partition; we will refer to such a hook as an {\em odd hook} of~$\lambda$. This combinatorial description of $f^n_k$ will be used throughout this paper, and we will regard $f^n_k$ also as a map between the corresponding sets of odd partitions. Also, for
$\mu \vdash_o n-2^k$ we set $e\big(\mu,2^k\big) = e\big(\chi^{\mu},2^k\big)$.
\end{rem}

We need some concepts and basic facts concerning hooks in partitions. For any integer $e\in\mathbb{N}$ we denote by $C_e(\lambda)$ and $Q_e(\lambda)$ the $e$-core and the $e$-quotient of $\lambda$, respectively. Then $Q_e(\lambda)=(\lambda_0,\ldots,\lambda_{e-1})$ is an $e$-tuple of partitions satisfying $n=|C_e(\lambda)|+e\sum\limits_{i=0}^{e-1}|\lambda_i|$. It is well-known that a partition is uniquely determined by its $e$-core and $e$-quotient (we refer the reader to~\cite{OlssonBook} or \cite[Chapter~2.7]{JK} for a detailed discussion on this topic).

Let $\mathcal{H}_e(\lambda)$ be the set of hooks of~$\lambda$ having length divisible by $e$, and let $\mathcal{H}(Q_e(\lambda))=\cup_{i=0}^{e-1} \mathcal{H}(\lambda_{i})$. As explained in \cite[Theorem~3.3]{OlssonBook}, there is a bijection between $\mathcal{H}_e(\lambda)$ and $\mathcal{H}(Q_e(\lambda))$ mapping hooks in $\lambda$ of length $ex$ to hooks in the quotient of length~$x$. Moreover, the bijection respects the process of hook removal. Namely, the partition $\mu$ obtained by removing a $ex$-hook from $\lambda$ is such that $C_e(\mu)=C_e(\lambda)$ and the $e$-quotient of $\mu$ is obtained by removing an $x$-hook from one of the partitions involved in $Q_e(\lambda)$.

For $e=2$ we want to repeat the process of taking 2-cores and 2-quotients to obtain the {\it $2$-quotient tower} $\mathcal{Q}_2(\lambda)$ and the {\it $2$-core tower} $\mathcal{C}_2(\lambda)$ of $\lambda$. They have rows numbered by $k \ge 0$. The $k$th row $\mathcal{Q}^{(k)}_2(\lambda)$ of $\mathcal{Q}_2(\lambda)$ contains $2^k$ partitions $\lambda^{(k)}_i$, $0 \le i \le 2^k-1$, and the $k$th row $\mathcal{C}^{(k)}_2(\lambda)$ of~$\mathcal{C}_2(\lambda)$ contains the $2$-cores of these partitions in the same order, i.e., $C_2\big(\lambda^{(k)}_i\big)$, \mbox{$0 \le i \le 2^k-1$}. The 0th row of $\mathcal{Q}_2(\lambda)$ contains $\lambda=\lambda^{(0)}_0$ itself, row~1 contains the partitions~$\lambda^{(1)}_0$,~$\lambda^{(1)}_{1}$ occurring in the $2$-quotient $Q_2(\lambda)$, row~2 contains the partitions occurring in the $2$-quotients of partitions occurring in row~1, and so on. Specif\/ically we have $Q_2\big(\lambda^{(k)}_i\big) =\big(\lambda^{(k+1)}_{2i},\lambda^{(k+1)}_{2i+1}\big)$ for $i \in \big\{0,1,\dots,2^k-1\big\}$. We remark that the $2^k$ partitions in $\mathcal{Q}^{(k)}_2(\lambda)$ are the same as those in the $2^k$-quotient $Q_{2^k}(\lambda)$ of~$\lambda$, but in a dif\/ferent order for~$k\ge 2$.

We also introduce the {\em $k$-data} $\mathcal{D}^{(k)}_2(\lambda)$ of $\lambda$. This is a table containing the following $k+1$ rows: the $k$ rows $\mathcal{C}^{(j)}_2(\lambda)$, $j=0,\dots, k-1$, and in addition the row $\mathcal{Q}^{(k)}_2(\lambda)$.

\begin{rem}\label{recover} A partition $\lambda$ may be recovered from its $2$-core tower. For $k>0$, it may also be recovered from the knowledge of the $k$-data $\mathcal{D}^{(k)}_2(\lambda)$ of $\lambda$, because the rows $\mathcal{C}^{(l)}_2(\lambda)$ with $l \ge k$ of~$\mathcal{C}_2(\lambda)$ consist of the $2$-core towers of the partitions in~$\mathcal{Q}^{(k)}_2(\lambda)$.
\end{rem}

\begin{lem}\label{hookremovedata} Suppose that $\lambda \vdash n-2^k$ and $\mu \vdash n$. The following are equivalent.
\begin{itemize}\itemsep=0pt
\item[$(i)$] $\lambda$ is obtained from $\mu$ by removing a $2^k$-hook.
\item[$(ii)$] The $k$-data $\mathcal{D}^{(k)}_2(\mu)$ and $\mathcal{D}^{(k)}_2(\lambda)$ coincide, except that for one $i \in \big\{ 0, \ldots , 2^k-1\big\}$ $\lambda^{(k)}_i$ is obtained from $\mu^{(k)}_i$ by removing a $1$-hook.
\end{itemize}
\end{lem}

\begin{proof}A $2^k$-hook $H_0$ in $\mu$ corresponds in a canonical way to a $2^{k-1}$-hook $H_1$ in a partition in~$\mathcal{Q}^{(1)}_2(\mu)$, i.e., in row 1 of the $2$-quotient tower~$\mathcal{Q}_2(\mu)$. Continuing we see that $H_0$ corresponds in a canonical way to a 1-hook $H_k$ in a~partition $\mu^{(k)}_i$ in $\mathcal{Q}^{(k)}_2(\mu)$, row $k$ of~$\mathcal{Q}_2(\mu)$. If $\lambda$ is obtained by removing $H_0$ from $\mu$, this corresponds to~$\lambda^{(k)}_i$ being obtained by removing the 1-hook $H_k$ from $\mu^{(k)}_i$ (by repeated applications of \cite[Theorem~3.3]{OlssonBook}). Apart from this the rows $\mathcal{Q}^{(k)}_2(\mu)$ and~$\mathcal{Q}^{(k)}_2(\lambda)$ coincide. Note also that the rows $\mathcal{C}^{(j)}_2(\mu)$ and $\mathcal{C}^{(j)}_2(\lambda)$ coincide for $j=0,\dots, k-1$, since the removal of the hooks~$H_j$ of even length do not change the 2-cores.
\end{proof}

Odd-degree characters of $\fS_n$ and thus odd partitions were completely described in~\cite{Mac}. We restate this result in a language which is convenient for our purposes. We let $c_2^{(k)}(\lambda)$ be the sum of the cardinalities of the partitions in the $k$th row $\mathcal{C}^{(k)}_2(\lambda)$ of $\mathcal{C}_2(\lambda)$.

\begin{lem}[\cite{Mac}]\label{macd} Let $\lambda$ be a partition. Then $\lambda$ is odd if and only if $c_2^{(k)}(\lambda) \leq 1$ for all $k\ge 0$.
\end{lem}

It may be decided from the $k$-data $\mathcal{D}^{(k)}_2(\lambda)$ whether $\lambda$ is odd. The case $k=1$ of the following result appeared in
\cite[Lemma~4.1]{INOT} and also in \cite[Lemma~6]{APS}.

\begin{thm}\label{thm:oddcriterion} Let $\lambda \vdash n$, and let $k \ge 0$ be fixed. Consider $\mathcal{Q}_2^{(k)}(\lambda)=\big(\lambda^{(k)}_i\big)$. Then $\lambda$ is odd if and only if the following conditions are all fulfilled:
\begin{itemize}\itemsep=0pt
\item[$(i)$] $c_2^{(j)}(\lambda) \leq 1$ for all $j < k$.
\item[$(ii)$] The partitions $\lambda^{(k)}_i$, $0\le i \le 2^k-1$, are all odd.
\item[$(iii)$] The numbers $\big|\lambda^{(k)}_i\big|$, $0\le i \le 2^k-1$, are pairwise $2$-disjoint.
\end{itemize}
In this case $ \sum\limits_{i \ge 0} \big|\lambda^{(k)}_i\big|=\big\lfloor \frac{n}{2^k} \big\rfloor$.
\end{thm}

\begin{proof}This is proved by induction on $k \ge 0$, using Remark~\ref{recover} and Lemma~\ref{macd}.
\end{proof}

We illustrate the result above by giving an example.

\begin{exa}\label{example1} Let $n=15$ and take $\lambda=\big(5,4,2^2,1^2\big) \vdash 15$. To decide whether $\lambda$ is odd, we choose $k=2$ and compute the 2-data $\mathcal{D}_2^{(2)}(\lambda)$. The 2-core is $C_2(\lambda)=(1)$, giving $\mathcal{C}_2^{(0)}(\lambda)=((1))$. Furthermore, the 2-quotient is $Q_2(\lambda)=\big(\big(2^2,1^2\big),(1)\big)$, and computing the 2-cores $C_2\big(\big(2^2,1^2\big)\big)=(0)$, $C_2((1))=(1)$, we obtain the next row: $\mathcal{C}_2^{(1)}(\lambda)=((0),(1))$. The 2-quotients are $Q_2\big(\big(2^2,1^2\big)\big)$ $=\big(\big(1^2\big), (1)\big)$, $Q_2((1))=((0), (0))$; hence the f\/inal row of the 2-data table is obtained as $\mathcal{Q}_2^{(2)}(\lambda)=\big(\big(1^2\big),(1),(0),(0)\big)$.

We visualize $\mathcal{D}_2^{(2)}(\lambda)$ like this:
\begin{gather*}
\begin{array}{@{}r c c c c c c c c c c c}
\mathcal{C}_2^{(0)}(\lambda)\colon &&&&(1)\\[5pt]
\mathcal{C}_2^{(1)}(\lambda)\colon &&(0)&&&&(1) \\[5pt]
\mathcal{Q}_2^{(2)}(\lambda)\colon &(1^2)&&(1)&&(0)&&(0)
\end{array}
\end{gather*}

Theorem~\ref{thm:oddcriterion} shows that $\lambda$ is odd and thus it contains a unique odd 4-hook. Again using the theorem, it is clear that removing this 4-hook corresponds to the second partition~$(1)$ in~$\mathcal{Q}_2^{(2)}$ being replaced by~$(0)$. Thus, removing the corresponding 4-hook of $\lambda$ we obtain the odd partition $\mu=\big(3,2^3,1^2\big) \vdash 11$ with the property that $\mathcal{D}_2^{(2)}(\lambda)$ and $\mathcal{D}_2^{(2)}(\mu)$ dif\/fer only in their f\/inal row.
\end{exa}

\begin{rem}\label{oddconstruct} Using the construction of partitions from their 2-cores and 2-quotients already mentioned, the criterion above can be applied to construct all odd partitions of $n$ with a specif\/ic $k$th row in the 2-quotient tower. For this, let $n,k\in \N$, and take any sequence of odd parti\-tions~$\nu_i$, $0\le i \le 2^k-1$, such that the numbers $|\nu_i|$ are pairwise 2-disjoint, and $ \sum\limits_{i \ge 0} |\nu_i|=\big\lfloor \frac{n}{2^k} \big\rfloor$. Then there are exactly $\prod\limits_{{m<k} \atop {2^m \subseteq_2 n}} 2^m$ odd partitions~$\lambda$ of~$n$ with $\mathcal{Q}_2^{(k)}(\lambda)=(\nu_i)$, obtained by choosing one 2-core in row~$m$ of the $k$-data table to be~$(1)$, for each $m<k$ such that $2^m \subseteq_2 n$.
\end{rem}

The following easy consequence of Theorem~\ref{thm:oddcriterion} will be used repeatedly.

\begin{lem}\label{Lemma 4} Let $2^t$ be the largest binary digit of $n$. A partition $\lambda$ of~$n$ is odd if and only if~$\lambda$ contains a~unique $2^t$-hook and the partition obtained from $\lambda$ by removing this $2^t$-hook is an odd partition of $n-2^t$.
\end{lem}

\section{Surjectivity and regularity}\label{sec:4}

The aim of this section is to study the images of the maps~$f_k^n$ for all~$n$,~$k$ such that $2^k\leq n$. For this purpose we introduce the concept of $d$\textit{-good partitions} (see Def\/inition~\ref{def:good} below). This will allow us to prove Theorem~\ref{thm5:image:reform} (describing the images) and thus Theorem~\ref{TheoremA} (describing exactly when $f_k^n$ is surjective) and to show that the maps $f_k^n$ are always regular on their image (see Corollary~\ref{cor:reg}).

\begin{defi}\label{def:good}
Let $d \ge 0$. We call an odd partition $\lambda$ {\it $d$-good}, if
\begin{itemize}\itemsep=0pt
\item[(i)] $|\lambda| \equiv 2^d-1 \mod 2^{d+1}$.
\item[(ii)] $C_{2^{d}}(\lambda) $ is a hook partition.
\end{itemize}
\end{defi}
Let us remark that condition (i) may be reformulated as
\begin{itemize}\itemsep=0pt
\item[(i$^*$)] $\nu_2(|\lambda|+1)=d$.
\end{itemize}

In particular, if $\lambda$ is $d$-good, then $|\lambda|$ is odd if and only if $d>0$.

The relevance of $d$-good partitions in our context is illuminated by the following reformulation of \cite[Theorem~2]{APS}:

\begin{lem}\label{ext:1} Let $\lambda \vdash_o n$. Let $d=\nu_2(n+1)$. Then $e(\lambda,1) \neq 0$ if and only if $\lambda$ is $d$-good. In this case, $e(\lambda,1)=1$ if $d=0$, and $e(\lambda,1)=2$ if $d>0$.
\end{lem}

\begin{lem}\label{d-good} Let $\lambda$ be an odd partition, and let $d\ge 0$. Then the following hold.
\begin{itemize}\itemsep=0pt
\item[$(1)$] For $d \le 2$, $\lambda$ is $d$-good if and only if $ |\lambda| \equiv 2^d-1 \mod 2^{d+1}$.
\item[$(2)$] If $\lambda$ is $d$-good, then $C_{2^{d}}(\lambda) $ is a partition of $2^d-1$.
\end{itemize}
\end{lem}
\begin{proof} If the odd partition $\lambda$ is $d$-good, then $|\lambda|=\big(2^d-1\big) +m $ where the binary digits of $m$ are at least $2^{d+1}$. The hooks of $\lambda$ corresponding to the binary digits of $m$ may be decomposed into $2^d$-hooks and thus do not contribute to $C_{2^{d}}(\lambda)$.
Thus $|C_{2^{d}}(\lambda)|=2^d-1$. This shows~(2). For $d=0,1,2$ we have $|C_{2^d}(\lambda)|=0$, $1$ and $3$, respectively. Since all partitions of 0, 1 and 3 are hook partitions, (1) follows.
\end{proof}

\begin{defi}\label{dnk} If $2^k \le n$, we def\/ine $ d(n,k)=\nu_2\big(\lfloor \frac{n}{2^k}\rfloor\big)$. Thus $d(n,k)$ is the smallest integer $d \ge 0$ satisfying the condition $2^{k+d}\subseteq_2 n$. In particular, $d(n,k)=0$ if and only if $2^{k}\subseteq_2 n$. Moreover, we may write $\big\lfloor \frac{n}{2^k}\big\rfloor=2^{d(n,k)}+m(n,k)$ where $2^{d(n,k)+1} \,|\, m(n,k)$.
\end{defi}

As mentioned in the introduction, the results in \cite{INOT} show that $f_k^n$ is a surjective $(2^k$-\text{to}-$1)$-map whenever $2^k\subseteq_2 n$, i.e., $d(n,k)=0$. In the spirit of \cite[Theorem~2]{APS}, we now give a characterization of the image of the map $f_k^n$ for all~$n$,~$k$ such that $2^k <n$.

\begin{thm}\label{thm5:image:reform} Let $n\in\mathbb{N}$, $k\in\mathbb{N}_0$ be such that $2^k<n$. Let $\lambda\vdash_o n-2^k$. Then $e\big(\lambda, 2^k\big)\neq 0$ if and only if there exists a $d(n,k)$-good partition in the $k$th row of $\mathcal{Q}_2(\lambda)$. In this case, $e\big(\lambda, 2^k\big)=2^k$ if $d(n,k)=0$, and $e\big(\lambda, 2^k\big)=2$ if $d(n,k)>0$.
\end{thm}

\begin{proof} If $k=0$ then the statement follows from Lemma~\ref{ext:1}. Hence assume that $k\geq 1$. Let $d=d(n,k)$. By assumption $\big\lfloor \frac{n}{2^k}\big\rfloor=2^d+m$, where the binary digits of $m$ are at least~$2^{d+1}$. Thus $\big\lfloor \frac{n-2^k}{2^k}\big\rfloor=\big(2^d-1\big)+m$.

Suppose f\/irst that $e\big(\lambda, 2^k\big)\neq 0$ and that $\mu \vdash_{\mathrm {o}} n$ satisf\/ies $f_k(\mu)=\lambda$. From Remark~\ref{rem:oddhook} and Lemma~\ref{hookremovedata} we get that there exists an $i\in\{0,1,\ldots, 2^k-1\}$ such that $f_0\big(\mu^{(k)}_i\big)=\lambda^{(k)}_i$. Since $\mu^{(k)}_i$ and $\lambda^{(k)}_i$ are odd, we get $e\big(\lambda^{(k)}_i,1\big) \ne 0$. We have that $\big|\lambda^{(k)}_i\big|$ and $\big|\mu^{(k)}_i\big|$ are both 2-disjoint with $m_1:= \sum\limits_{j \ne i} \big|\lambda^{(k)}_j\big|=\sum\limits_{j \ne i} \big|\mu^{(k)}_j\big| \subseteq_2 \big\lfloor \frac{n-2^k}{2^k} \big\rfloor$, by Theorem~\ref{thm:oddcriterion}. Since $m_1 \subseteq_2 \big\lfloor \frac{n-2^k}{2^k}\big\rfloor$ and $m_1 \subseteq_2 \big\lfloor \frac{n}{2^k}\big\rfloor$, we get $m_1 \subseteq_2 m$. Thus $\big|\lambda^{(k)}_i\big|=(2^d-1)+m_2$ and $\big|\mu^{(k)}_i\big|=2^d+m_2$, where $m_2=m-m_1 \subseteq_2 m$. In particular $\nu_2\big(\big|\lambda^{(k)}_i\big|+1\big)=\nu_2\big(\big|\mu^{(k)}_i\big|\big)=d$. Then Lemma~\ref{ext:1} shows that~$\lambda^{(k)}_i$ is $d$-good.

Conversely, if $\lambda^{(k)}_i$ is a $d$-good partition for some $i\in\big\{0,1,\ldots, 2^k-1\big\}$, then there exists a~$\mu^*\vdash_o \big|\lambda^{(k)}_i\big|+1$ such that $f_0(\mu^*)=\lambda^{(k)}_i$, by Lemma~\ref{ext:1}. We let $\mu$ be the partition where the $k$-data $\mathcal{D}^{(k)}_2(\mu )$ and $\mathcal{D}^{(k)}_2(\lambda)$ coincide, except that $\mu^{(k)}_i=\mu^*$. Since $\lambda$ is odd and~$\lambda^{(k)}_i$ is $d$-good, we know that $\big|\lambda^{(k)}_i\big|=\big(2^d-1\big)+m'$ where $m' \subseteq_2 m$, and $\big|\lambda^{(k)}_j\big| \subseteq_2 m-m'$ for all $j\ne i$. Hence $|\mu^*|= \big|\lambda^{(k)}_i\big|+1=2^d+m'$ is 2-disjoint from all $\big|\lambda^{(k)}_j\big|$, $j\ne i$. Thus~$\mu$ is an odd partition of $n$ by Theorem~\ref{thm:oddcriterion}, and $f_k(\mu)=\lambda$ by Lemma~\ref{hookremovedata} and Remark~\ref{rem:oddhook}.

We conclude that $e\big(\lambda,2^k\big)=\sum\limits_{\lambda^{(k)}_i d-{\rm good}} e\big(\lambda^{(k)}_i,1\big)$. If $d=0$ then $\big\lfloor \frac{n-2^k}{2^k}\big\rfloor$ is even. This implies that all $\lambda^{(k)}_i$ are of even cardinality and thus $d$-good. Thus $e\big(\lambda^{(k)}_i,1\big)=1$ for all $i$, and we get $e\big(\lambda, 2^k\big)=2^k$. If $d>0$ there is exactly one $\lambda^{(k)}_i$ in $\mathcal{Q}^{(k)}_2(\lambda)$ of odd cardinality. Only this $\lambda^{(k)}_i$ may be $d$-good and then $e\big(\lambda,2^k\big)=e\big(\lambda^{(k)}_i,1\big)=2$. Otherwise $e\big(\lambda,2^k\big)=0$.
\end{proof}

\begin{cor}\label{thm5:image:reform-nogood}
Let $n\in\mathbb{N}$, $k\in\mathbb{N}_0$ be such that $2^k<n$, and let $d=\nu_2\big(\big\lfloor \frac{n}{2^k}\big\rfloor\big)$. Let $\lambda\vdash_o n-2^k$. Then $e\big(\lambda, 2^k\big)\neq 0$ if and only if there exists a partition $\lambda^{(k)}_i$ in the $k$th row of $\mathcal{Q}_2(\lambda)$ such that $\big|\lambda^{(k)}_i\big| \equiv 2^d-1 \mod 2^{d+1}$, and $C_{2^d}\big(\lambda^{(k)}_i\big)$ is a hook partition. In this case, $e\big(\lambda, 2^k\big)=2^k$ if $d=0$, and $e\big(\lambda, 2^k\big)=2$ if $d>0$.
\end{cor}

We are now ready to prove Theorem~\ref{TheoremA}. In fact, this is a consequence of Theorem~\ref{thm5:image:reform} and it is stated here as the following corollary.

\begin{cor}[Theorem~\ref{TheoremA}]\label{surj} Let $n\in\mathbb{N}$, $k\in\mathbb{N}_0$ be such that $2^k<n$.
\begin{itemize}\itemsep=0pt
\item If $k=0$ then $f^n_k$ is surjective if and only if $d(n,k) \le 2$.
\item If $k>0$ then $f^n_k$ is surjective if and only if $d(n,k) \le 1$.
\end{itemize}
\end{cor}
\begin{proof} By Theorem \ref{thm5:image:reform}, $f^n_k$ is surjective if and only if for all $\lambda \vdash_o n-2^k$ we have that the $k$th row of $\mathcal{Q}_2(\lambda)$ contains a~$d(n,k)$-good partition $\lambda^{(k)}_j$. By Theorem~\ref{thm:oddcriterion} and Def\/inition~\ref{dnk}, for any
$\lambda \vdash_o n-2^k$ we have $ \sum\limits_{j \ge 0} \big|\lambda^{(k)}_j\big|=\big\lfloor \frac{n-2^k}{2^k} \big\rfloor = \big(2^{d(n,k)}-1\big)+m(n,k)$.

If $k=0$ then $\mathcal{Q}^{(0)}_2(\lambda)$ contains only $\lambda=\lambda^{(0)}_0$. Hence $f^n_0$ is surjective if and only all odd partitions of $n-1$ are $d(n,0)$-good. By Lemma~\ref{d-good}(1), the latter condition holds when $d=d(n,0) \le 2$. On the other hand, if $d=\nu_2(n)>2$, then $\lambda=(n-5,2,2)$ is an odd partition of $n-1$ by Theorem~\ref{thm:oddcriterion}, but $C_8(\lambda)=(3,2,2)$ is not a hook, and hence $C_{2^d}(\lambda)$ is not a~hook. So~$\lambda$ is not $d$-good, and thus $f^n_0$ is not surjective.

Now assume $k\geq 1$. Then $\mathcal{Q}^{(k)}_2(\lambda)$ contains at least two odd partitions. If $d(n,k) \ge 2$ then any $d(n,k)$-good partition $\mu$ satisf\/ies $3 \subseteq_2 2^{d(n,k)}-1 \subseteq_2 |\mu|$. Write $\big\lfloor \frac{n-2^k}{2^k} \big\rfloor =1+m_1$ where $m_1$ is even. Applying Remark~\ref{oddconstruct}, take any $\lambda\vdash_o n-2^k$ such that $\big|\lambda^{(k)}_0\big|=1$ and $\lambda^{(k)}_1$ is an odd partition with $\big|\lambda^{(k)}_1\big|=m_1$. Then no partition in $\mathcal{Q}^{(k)}_2(\lambda) $ is $d(n,k)$-good. Thus $f^n_k$ is not surjective. On the other hand, if $d(n,k)=0$ then $2^k \subseteq_2 n$ and $f^n_k$ is surjective \cite[Proposition~4.5]{INOT}. If $d(n,k)=1$ then $\big\lfloor \frac{n-2^k}{2^k} \big\rfloor = 1+m(n,k)$, where $4 \mid m(n,k)$. Thus any $\mathcal{Q}^{(k)}_2(\lambda) $ contains a partition with odd cardinality; this partition is 1-good, by Lemma~\ref{d-good}. Again $f^n_k$ is surjective.
\end{proof}

It is an immediate consequence of Theorem \ref{thm5:image:reform} that $f_k^n$ is regular on its image for all relevant choices of $n,k$ such that $2^k < n$. We have:

\begin{cor}\label{cor:reg}
Let $n\in\mathbb{N}$, $k\in\mathbb{N}_0$ be such that $2^k < n$; set $d=\nu_2\big(\big\lfloor \frac{n}{2^k}\big\rfloor\big)$. Let $\lambda\vdash_o n-2^k$. Then
\begin{gather*}
e\big(\lambda, 2^k\big)= \begin{cases}
2^k & \text{if } d=0;\\
2 &\text{if } d>0,\ \text {and the}\ k\text{th row of}\ \mathcal{Q}_2(\lambda) \text{ contains a } d\text{-good partition};\\
0 &\text{otherwise. }\\
\end{cases}
\end{gather*}
\end{cor}

\begin{exa}\label{example2} For an illustration, we consider odd extensions of odd partitions by a $4$-hook, i.e., we take $k=2$ above. For $n>2^2$ we f\/irst compute $d(n,k)=\nu_2\big(\big\lfloor \frac{n}{2^k}\big\rfloor\big)$, and then consider odd partitions of $n-4$ and their 4-extensions. For $n=6$, $d(6,2)=0$. Thus $e((2),4)=4$. The odd 4-extensions of $(2)$ are $(6)$, $\big(3^2\big)$, $\big(2^2,1^2\big)$, $\big(2,1^4\big)$.
For $n=10$, $d(10,2)=1$. In this case, $e(\lambda,4)=2$ for all odd partitions $\lambda$ of 6. For instance, the odd 4-extensions of $(6)$ are $(10)$ and $(6,3,1)$. For $n=19$, $d(19,2)=2$. Example~\ref{example1} shows that for $\lambda=\big(5,4,2^2,1^2\big) \vdash_o 15$
there is no 2-good partition in $\mathcal{Q}_2^{(2)}(\lambda)$, hence $e(\lambda,4)=0$.
\end{exa}

\section[Deciding commutativity of the maps $f_k$ and $f_{\ell}$]{Deciding commutativity of the maps $\boldsymbol{f_k}$ and $\boldsymbol{f_{\ell}}$}

Let $n\in\mathbb{N}$, and suppose that $0\le k <\ell$ satisfy $2^k+2^{\ell}\leq n$. As stated in the introduction, we want to complete the discussion of the commutativity of the maps $f_k$ and $f_\ell$. Since the relevant~$n$ will always be apparent for the maps $f^n_k$ in this section, we just write~$f_k$.
	
We write $(n;k,\ell) \in \mathcal{T} $ if for all $\lambda \vdash_o n$ we have $f_kf_{\ell}(\lambda)=f_{\ell}f_k(\lambda)$. Otherwise we write $(n;k,\ell) \in \mathcal{F}$.

In this section we will prove Theorem~\ref{TheoremB}, which may be reformulated as follows.

\begin{thm}\label{them:Aref}
Let $n=2^t+m$ where $0 \le m < 2^t$. Suppose that $k$, $\ell$ satisfy $0 \leq k<\ell $ and $2^k+2^{\ell}\leq n$. Then {\it with the exception of} $(6;0,1)$
\begin{gather*}
\text{$(n;k,\ell) \in \mathcal{F}$ if and only if $\ell <t$ and $2^k \leq m$}.
\end{gather*}
\end{thm}

The proof of Theorem \ref{them:Aref} is based on a series of lemmas. The f\/irst lemmas concern two \textit{extreme} cases, where $f_k$ and~$f_{\ell}$ commute.

In the case $\ell=t$ we have the following result as a~reformulation of \cite[Proposition~4.3]{INOT}.

\begin{lem}\label{Lemma 1}
Let $n=2^t+m$ with $0 \le m < 2^t$. If $2^k \leq m$, then $(n; k, t) \in \mathcal{T}$.
\end{lem}

It is also known that in the case where $n$ is a power of $2$, the maps $f_k$ and $f_{\ell}$ commute \cite[Remark~4.4]{INOT}, and we include a short proof here.

\begin{lem}\label{Lemma 2} If $n=2^t $ then $(n;k,\ell) \in \mathcal{T}$ for all $k$, $\ell$.
\end{lem}

\begin{proof} If $0 \le b \le a$ are integers then the binomial coef\/f\/icient $\binom{a}{b}$ is odd if and only if $b\subseteq_2 a$, by Lucas' theorem. The odd partitions of $2^t$ are exactly the hook partitions $\big(2^t-b, 1^b\big)$, $0 \le b \le 2^t-1$, of degree $\binom{2^t-1}{b}$. Hence for $k \in \{0,1,\ldots,t-1\}$ we have
\begin{gather*}
f_k(\lambda)= \begin{cases}
\big(2^t-b-2^k, 1^b\big) & \text{if } 2^k \not\subseteq_2 b,\\
\big(2^t-b, 1^{b-2^k}\big) & \text{if } 2^k \subseteq_2 b.
\end{cases}
\end{gather*}
It follows that for any $k,\ell<t$ and odd partition $\lambda$ of $2^t$, we have $f_\ell f_k(\lambda)=f_k f_\ell(\lambda)$.
\end{proof}

\begin{lem}\label{Lemma 3} Let $n=2^t+m$ with $0 \le m < 2^t$. Suppose that $k$, $\ell$ satisfy $0 \leq k<\ell$ and $2^k+2^{\ell}\leq n$. If $m<2^k$ then $(n; k, \ell) \in \mathcal{T}$.
\end{lem}

\begin{proof} We use induction on $k\ge 0$. For $k=0$ we have $m=0$ and the claim follows from Lemma~\ref{Lemma 2}. Suppose that $k \ge 1$ and that the claim has been proved up to $k-1$. Let $\lambda \vdash_o n$. Odd hooks of length $2^k$ and $2^{\ell}$ in $\lambda$ correspond to odd hooks of length $2^{k-1}$ and $2^{\ell-1}$ in the $2$-quotient $Q_2(\lambda)=(\lambda_0, \lambda_1)$ of $\lambda$. From Theorem~\ref{thm:oddcriterion} we deduce that $|\lambda_0|$ and $|\lambda_1|$ are 2-disjoint binary subsums of $\big\lfloor \frac{n}{2}\big\rfloor$, so one of them contains $2^{t-1}$, say $|\lambda_0|$; then $|\lambda_1| \le \big\lfloor \frac{m}{2}\big\rfloor < 2^{k-1}<2^{\ell-1}$. Thus the odd $2^{k-1}$-hook in $Q_2(\lambda)$ has to be in $\lambda_0$. Therefore
\begin{gather*}
Q_2(f_k(\lambda))=(f_ {k-1}(\lambda_0), \lambda_1).
\end{gather*}
Applying $f_{\ell}$, the odd $2^{\ell-1}$-hook cannot be in $\lambda_1$, hence
\begin{gather*}
Q_2(f_{\ell}f_k(\lambda))=(f_{\ell -1} f_ {k-1}(\lambda_0), \lambda_1)).
\end{gather*}	
In particular, we know that $|\lambda_0| \ge 2^{\ell -1}+2^{k-1}$. Also $|\lambda_0|+|\lambda_1|=\big\lfloor \frac{n}{2} \big\rfloor=2^{t-1}+\big\lfloor \frac{m}{2} \big\rfloor$. We have already seen that $2^{t-1}$ is the largest binary digit of $|\lambda_0|$; furthermore $|\lambda_0|-2^{t-1}$ is a~binary subsum of $\big\lfloor \frac{m}{2}\big\rfloor <2^{k-1}$. We may therefore apply the inductive hypothesis to $\lambda_0$ to get $f_{\ell -1} f_ {k-1}(\lambda_0)=f_{k-1}f_{\ell-1}(\lambda_0)$. This implies that $Q_2(f_k f_{\ell}(\lambda))= Q_2(f_{\ell} f_k(\lambda))$ and thus $f_k f_{\ell}(\lambda)=f_{\ell}f_k(\lambda)$.
\end{proof}

Lemmas \ref{Lemma 1} and~\ref{Lemma 3} show that the \textit{only if} part of the theorem is true. We now turn to the \textit{if} part. We start by proving the statement for $k=0$ and use this as part of an inductive argument.

\begin{lem}\label{Lemma 5} Let $n=2^t+m$ with $0 < m < 2^t$. If $0<\ell<t$ then $(n; 0, \ell) \in \mathcal{F}$, with the exception of $(6;0,1)$.
\end{lem}

\begin{proof} The result is easily checked for $n \le 8$, which includes the exception $(6;0,1)$. So we assume that $t \ge 3$.

{\it Case 1}: $2^{\ell} < m$. Then $m\ge 3$, since $\ell >0$. Consider the partition $\lambda=\big(m,m,1^{a}\big) \vdash n$ where $a=n-2m=2^t-m$. The (1,1)-hook length of $\lambda$ is $2^t+1$. The (2,1)-hook length of $\lambda$ is~$2^t$. Removing the (2,1)-hook hook we get the odd partition~$(m)$, so $\lambda$ is odd, by Lemma~\ref{Lemma 4}. We claim that
\begin{gather*}
f_0(\lambda)=\big(m,m,1^{a-1}\big).
\end{gather*}
Indeed we cannot have $f_0(\lambda)=\big(m,m-1,1^a\big)$ because this partition does not have a hook of length $2^t$, and thus it is not odd. Now
\begin{gather*}
f_{\ell}(f_0(\lambda))=f_{\ell}\big(m,m,1^{a-1}\big)=\big(m,m-2^{\ell},1^{a-1}\big)
\end{gather*}
since $\big(m,m,1^{a-1-2^{\ell}}\big)$ and $\big(m-1,m-2^{\ell}+1,1^{a-1}\big)$ both do not have a hook of length $2^t$ and thus are not odd (again by Lemma~\ref{Lemma 4}).

On the other hand,
\begin{gather*}
f_{\ell}(\lambda)=\big(m-1,m-\big(2^{\ell}-1\big),1^{a}\big).
\end{gather*}
Indeed, the other candidates for $f_{\ell}(\lambda)$, which are $\big(m,m-2^{\ell},1^a\big)$ and $\big(m,m,1^{a-2^{\ell}}\big)$, do not have hooks of length $2^t$. Then
\begin{gather*}
f_0(f_{\ell}(\lambda))=f_0\big(m-1,m-\big(2^{\ell}-1\big),1^{a}\big)=\big(m-1, m-2^{\ell},1^a\big).
\end{gather*}
This follows (again) by observing that all the other partitions of $n-2^\ell-1$ obtained from $\big(m-1,m-\big(2^{\ell}-1\big),1^{a}\big)$ by removing a node do not have hooks of length $2^t$. Thus $f_0(f_{\ell}(\lambda)) \neq f_{\ell}(f_0(\lambda))$.

{\it Case 2}: $m < 2^{\ell}$. Consider the partition $\lambda=\big(n-2^{\ell}, m+1,1^a\big)$, where $a=2^{\ell}-(m+1)$. Note that $n-2^{\ell} \ge m+1$ since $\ell<t$ by assumption, and that $a \ge 0$. The (1,1)-hook length of $\lambda$ is $n-m=2^t$. Removing this hook we get the odd partition $(m)$, so $\lambda$ is odd. The (2,1)-hook length of $\lambda$ is $2^{\ell}$. Now
\begin{gather*}
f_0(\lambda)=\big(n-2^{\ell}, m,1^a\big)
\end{gather*}
since the other candidates do not have hooks of length $2^t$. Then
\begin{gather*}
f_{\ell}(f_0(\lambda))=f_{\ell}\big(n-2^{\ell}, m,1^a\big)=\mu,
\end{gather*}
where $\mu$ is obtained from $f_0(\lambda)$ by removing a~$2^{\ell}$-hook in the f\/irst row. (There are only hooks of length $<2^{\ell}$ in the other rows.) In fact, $\mu=\big(n-2^{\ell+1},m, 1^a\big)$ since $n-2^{\ell+1}\ge n-2^{t}=m$. Thus $f_{\ell}(f_0(\lambda))$ has at least 2 parts.
On the other hand
\begin{gather*}
f_{\ell}(\lambda)=\big(n-2^{\ell}\big)
\end{gather*}
since this odd partition is obtained from the odd partition $\lambda$ by removing a~$2^{\ell}$-hook (the one in~$(2,1)$). It follows that
\begin{gather*}
f_0(f_{\ell}(\lambda))=\big(n-2^{\ell}-1\big)
\end{gather*}
and again $f_0(f_{\ell}(\lambda)) \neq f_{\ell}(f_0(\lambda))$.

{\it Case 3}: $m=2^{\ell}$. Then $n=2^t+2^{\ell}$. If $\ell\geq 2$ then choose $\lambda=\big(2^t,2^{\ell}-1,1\big)$. The $(1,2)$-hook length of $\lambda$ is $2^t$; thus $\lambda$ is an odd partition since removing this $2^t$-hook gives an odd partition $\big(2^{\ell}-2,1,1\big)$ of $2^{\ell}$. We have $f_0(\lambda)=\big(2^t, 2^{\ell}-2,1\big)$ since the other candidates are not odd. Then
\begin{gather*}
f_{\ell}(f_0(\lambda))=\big(2^t-2^{\ell},2^{\ell}-2,1\big).
\end{gather*}
The $(2,1)$-hook length of $\lambda$ is $2^{\ell}$, so $f_{\ell}(\lambda)=\big(2^t\big)$ and
\begin{gather*}
f_0(f_{\ell}(\lambda))=\big(2^t-1\big),
\end{gather*}
showing $f_0(f_{\ell}(\lambda)) \neq f_{\ell}(f_0(\lambda))$.

On the other hand, if $\ell=1$ then choose $\lambda=\big(2^t-2,2,2\big)\vdash_o 2^t+2=n$. Since $t\geq 3$, it is now easy to show that $f_1(f_0(\lambda))=\big(2^t-4,2,1\big)$. On the other hand we see that $f_0(f_1(\lambda))$ is a~hook partition of $2^t-1=n-3$ and therefore is not equal to $f_1(f_0(\lambda))$.
\end{proof}

\begin{lem}\label{Lemma 6} If $(n; k, \ell) \in \mathcal{F}$ then also $(2n; k+1, \ell+1) \in \mathcal{F}$ and $(2n+1; k+1, \ell+1) \in \mathcal{F}$.
\end{lem}

\begin{proof} Let the odd partition $\mu$ of $n$ satisfy $f_kf_{\ell}(\mu) \neq f_{\ell}f_k(\mu)$. Let $\lambda$ be a partition of $2n$ or $2n+1$ having 2-quotient $Q_2(\lambda)=(\mu, (0))$. Then $\lambda$ is odd, by Theorem~\ref{thm:oddcriterion}. We have
\begin{gather*}
Q_2(f_{k+1}f_{\ell +1}(\lambda))=(f_kf_{\ell}(\mu),(0)) \neq (f_{\ell}f_k(\mu),(0)) =Q_2(f_{\ell+1}f_{k +1}(\lambda)),
\end{gather*}
so that
$f_{k+1}f_{\ell +1}(\lambda) \neq f_{\ell+1}f_{k +1}(\lambda)$.
\end{proof}

We are now ready to conclude this section with the proof of Theorem~\ref{TheoremB}.

\begin{proof}[Proof of Theorem~\ref{them:Aref}] The \textit{only if} part follows from Lemmas~\ref{Lemma 1} and~\ref{Lemma 3}. To prove the \textit{if} part we use induction on $k \ge 0$. If $k=0$, then the statement follows from Lemma~\ref{Lemma 5}. Let $k>1$ and suppose that the assertion is true up to and including $k-1$. To show that $(n; k, \ell) \in \mathcal{F}$ it suf\/f\/ices to prove $(\lfloor \frac{n}{2} \rfloor; k-1,\ell-1) \in \mathcal{F}$, by Lemma~\ref{Lemma 6}. We are assuming $n=2^t+m$, $0 \le m < 2^t$, $0 \leq k<\ell \leq t$ and $2^k+2^{\ell}\leq n$.
This implies $\big\lfloor \frac{n}{2} \big\rfloor =2^{t-1}+ \big\lfloor \frac{m}{2} \big\rfloor$, $0 \le \big\lfloor \frac{m}{2} \big\rfloor < 2^{t-1} $ and $2^{k-1}+2^{\ell-1}\leq \big\lfloor \frac{n}{2} \big\rfloor $. We may apply the inductive hypothesis to get $\big(\big\lfloor \frac{n}{2} \big\rfloor; k-1,\ell-1\big) \in \mathcal{F}$, and then $(n; k, \ell) \in \mathcal{F}$ except when $\big(\big\lfloor \frac{n}{2} \big\rfloor; k-1,\ell-1\big)=(6;0,1)$. In that case we are considering (12;1,2) or (13;1,2) which are both in~$ \mathcal{F}$, by direct computation (consider for example $(6,4,2)\vdash_o 12$ and $(6,4,3)\vdash_o 13$, respectively).
\end{proof}

\subsection*{Acknowledgements}
We thank Gabriel Navarro for providing some useful tables which helped us f\/ind the patterns explained by the results of this paper. The second author is grateful to Trinity Hall, University of Cambridge, for funding his research. Thanks also go to the referee for helpful comments on a previous version of the article.

\pdfbookmark[1]{References}{ref}
\LastPageEnding

\end{document}